\documentclass[12pt,a4paper]{article}

\usepackage{xcolor}
\usepackage{chngcntr}
\usepackage{amsmath}
\usepackage{amssymb}
\usepackage{amsthm}
\usepackage{enumitem}
\usepackage{setspace}
\usepackage{xypic}
\usepackage{graphicx}
\usepackage[nottoc,notlot,notlof]{tocbibind}
\usepackage{authblk}
\usepackage[top=2cm, bottom=2cm, left=2cm, right=2cm]{geometry}
\usepackage{hyperref}

\input xy
\xyoption{all}

\DeclareMathOperator{\Gal}{Gal}

\newtheorem{prop}{Proposition}[section]
\newtheorem{lem}[prop]{Lemma}
\newtheorem{theorem}[prop]{Theorem}
\theoremstyle{definition}
\newtheorem{defu}[prop]{Definition}
\newtheorem*{defu*}{Definition}
\newtheorem*{notation*}{Notation and conventions}
\newtheorem{rem}[prop]{Remark}
\newtheorem{ex}[prop]{Example}
\newtheorem{cor}[prop]{Corollary}

\newcommand{\normal}{\triangleleft}

\newcommand{\solv}{\text{solv}}

\newcommand{\C}{\mathcal{C}}

\newcommand{\res}{\textnormal{res}}

\newcommand{\Mgal}{\tilde M}
\newcommand{\Ehat}{\hat{E}}

\newcommand{\id}{\text{id}}

\newcommand{\cmmnt}[1]{\ignorespaces}

\counterwithin*{equation}{section}

\begin{document}
\pagenumbering{arabic}
\setcounter{page}{1}

\title{$\Theta$-Hilbertianity and strong $\Theta$-Hilbertianity%
\thanks{Research supported by ISF grant No.~696/13.}}
\author{S. Fried\thanks{A part of a Ph.D. Thesis done at Tel Aviv University under the supervision of the second author.} \and D. Haran}
\date{} 
\maketitle


\abstract{$\Theta$-Hilbertianity and its strengthening, strong $\Theta$-Hilbertianity, are two generalizations of Hilbertianity inspired by Jarden's definition of $p$-Hilbertianity and strong $p$-Hilbertianity. Jarden has asked whether the two notions defined by him are actually the same. We address this question in its more general version of $\Theta$-Hilbertianity and show that for PRC, and, in particular, for PAC fields, $p$-Hilbertianity and strong $p$-Hilbertianity coincide.}

\section{Introduction}

A field $K$ is called Hilbertian if the following condition holds: For every irreducible polynomial in two variables $f(t,X)\in K[t,X]$, separable in $X$, there exist infinitely many $a\in K$ such that $f(a, X)$ is irreducible in $K[X]$. Fields with this property are called Hilbertian because of Hilbert's Irreducibility Theorem (\cite[Satz~IV]{Hilbert}): 
Number fields are Hilbertian.


The paramount importance of Hilbertian fields lies in applications to Galois theory.
Namely, $t \mapsto a$ defines a $K$-place
of the field of rational functions $K(t)$ with residue field $K$.
If $f(t,X)$ generates a Galois extension $F$ of $K(t)$ with Galois group $G$,
then this place extends to a $K$-place of $F$;
let $F_a$ denote its residue field.
Now, if $f(a,X)$ is irreducible over $K$,
then $[F:K(t)]=[F_a:K]$.
In particular, the Galois group of $F_a/K$
is then isomorphic to $G$.
(If $f(a,X)$ is reducible, then, excluding finitely many $a \in K$, the Galois group of $F_a/K$ is isomorphic to a subgroup of $G$, uniquely determined by $a$ up to conjugation.)
In fact (see Section~\ref{sec: preliminaries}),
$K$ is Hilbertian
if and only if for every finite Galois extension $F/K(t)$ 
there are infinitely many $a \in K$ for which  there exists a $K$-place $\varphi\colon F \to F_a\cup\{\infty\}$ with $\varphi(t) = a$ and $[F_a : K]= [F:K(t)]$.

Naturally, such an important notion calls for generalizations.
These are many and the literature abounds:
\begin{itemize}
    \item
  Brauer-Hilbertian fields of Fein, Saltman, and Schacher (\cite{FSS}),  
    \item 
the $0$-Hilbertian fields of Corvaja and Zannier (\cite[Section 13.5]{FA}),
generalized to
    \item
$g$-Hilbertian fields
by Fried and Jarden (\cite[Section 13.5]{FA}),
based on the alternative definition of Hilbertian fields
by thin sets of Sansuc and Colliot-Th\'el\`ene
(see \cite[Section 13.5]{FA}),
    \item
variants such as RG-Hilbertian fields
introduced by Fried and V\"olklein in \cite{FV}
and investigated by D\`ebes and the second author in \cite{DH},
    \item
real Hilbertian fields of \cite{FHV},
 \item
 fully Hilbertian fields of Bary-Soroker and Paran (\cite{BSP}),
    \item
$\Sigma$-Hilbertian fields of Fried and Jarden (\cite{FJSigma}),
and
    \item
$p$-Hilbertian fields of Jarden (\cite{p-Hilbert}),
generalized to
    \item
$\Theta$-Hilbertian fields in \cite{FH2}.
\end{itemize}

The last three examples consider the above mentioned
specialization property of Galois groups
and require that it hold 
only for certain groups. This is very natural. For example, let $K$ be a Hilbertian field and $p$ a prime number. Consider $K_p$, the fixed field of a $p$-Sylow subgroup of $\Gal(K)$. Clearly, $K_p$ cannot be Hilbertian since it does not admit Galois extensions with Galois groups that are not $p$-groups. Nevertheless, Jarden \cite{p-Hilbert} proves that $K_p$ has the specialization property for $p$-groups and calls fields with this property \textbf{$p$-Hilbertian}. The proof actually shows that $K_p$ possesses a seemingly stronger property, which Jarden calls \textbf{strong $p$-Hilbertianity}: A finite Galois extension with an arbitrary group can be specialized to one of its $p$-Sylow subgroups.

Similarly, if $K$ is a Hilbertian field, then $K_\solv$, the maximal pro-solvable extension of $K$, has no proper solvable extensions and therefore cannot be Hilbertian. But it has the specialization property for groups that have no nontrivial solvable quotient
(\cite[Theorem 2.6]{FH2}). Moreover, any finite group $G$ can be specialized to its subgroup $G^\solv$, the kernel of the quotient map from $G$ onto its maximal solvable quotient.

The main insight leading to the definition of \textbf{$\Theta$-Hilbertianity} and of its seemingly stronger version \textbf{strong $\Theta$-Hilbertianity} \cite[Definition 2.1]{FH2} is that both the $p$-Sylow case as well as the solvable case may be treated simultaneously by using \textbf{Sylowian maps} -- maps that assign to every profinite group $G$ the conjugacy class of some closed subgroup of $G$ such that certain conditions are satisfied. Every Sylowian map $\Theta$ gives rise to a class $\C(\Theta)$ of finite groups and, in essence, $\Theta$-Hilbertianity means that the specialization property holds for all Galois groups $G$ such that $G\in\C(\Theta)$, while strong $\Theta$-Hilbertianity means that every finite Galois group $G$ can be specialized to one of its subgroups that belong to $\Theta(G)$. This is indeed a strengthening, since $\C(\Theta)$ consists of all finite groups $G$ such that $\Theta(G)=\{G\}$.

Jarden (\cite{p-Hilbert}) asks whether every $p$-Hilbertian field with pro-$p$ absolute Galois group is strongly $p$-Hilbertian. In this work we address this question in its more general version of $\Theta$-Hilbertianity. We show that for PAC fields whose absolute Galois group is pro-$\C(\Theta)$ and, under certain additional conditions, also for PRC fields, $\Theta$-Hilbertianity and strong $\Theta$-Hilbertianity coincide. From this we deduce a positive answer to Jarden's question for PAC and PRC fields with pro-$p$ absolute Galois groups. Whether this is true in the general case remains an open question.

\bigskip

This paper is structured as follows:
In Section \ref{sec: preliminaries} we recall the necessary definitions
and make the technical preparations.
Section \ref{Main Results} contains our main results.

\medskip
\begin{notation*}
For a field $K$ we denote by $K_s$ its separable closure. If $L/K$ is a Galois extension of fields, we denote by $\Gal(L/K)$ its Galois group; we denote by $\Gal(K)$ the absolute Galois group $\Gal(K_s/K)$ of $K$.

Groups in this work are tacitly assumed to be profinite groups, their subgroups are assumed to be closed and all the homomorphisms between profinite groups are continuous. 

If $\varphi$ is a place of a field $F$ and $E$ is a subfield of $F$, we denote by $\overline{E}$ the residue field of $E$ at the restriction of $\varphi$ to $E$ (omitting the reference to $\varphi$, which will be clear from the context). For a field $K$ we denote by $K(t)$ the field of rational functions in one variable over $K$. 
\end{notation*}

\paragraph{Acknowledgement}
We wish to express our gratitude to Lior Bary-Soroker for inspiring conversations and to the anonymous referees for their valuable remarks and suggestions that helped us to improve this work.

\section{Preliminaries}\label{sec: preliminaries}

Unless otherwise stated, the content of this section is taken from \cite{FH1} and \cite{FH2}.

\subsection{Sylowian maps and \texorpdfstring{$\Theta$}{T}-Hilbertianity}

Sylowian maps are used to define $\Theta$-Hilbertianity, a generalization of Hilbertianity.

\begin{defu}\label{def; Sylowian map}
Let $\Theta$ be a map that assigns to every profinite group $G$ the conjugacy class $\Theta(G)$ of a closed subgroup of $G$. We call the map $\Theta$ \textbf{Sylowian} if the following two conditions are satisfied:
\begin{enumerate}
\item [(a)] Let $\varphi \colon G \to H$ be an epimorphism of profinite groups. Then $\varphi(\Theta(G))=\Theta(H)$.
\item [(b)] Assume $U\in \Theta(G)$. Then $\Theta(U)=\{U\}$.
\end{enumerate}
If $F/E$ is a Galois extension, denote by $\Theta(F/E)$ the set of
intermediate fields $E \subseteq E' \subseteq F$ such that $\Gal(F/E')\in\Theta(\Gal(F/E))$.
If $F=E_s$, write $\Theta(E)$ instead of $\Theta(F/E)$.
\end{defu}

\begin{ex}\label{examples of Sylowian maps}
\begin{enumerate}
\item [(a)] The trivial (resp. identity) map $\Theta(G)=\{1\}$ (resp. $\Theta(G)=\{G\}$) for all profinite groups $G$ is a  \cmmnt{consistent} Sylowian map.
\item [(b)] Let $p$ be a prime number. For a profinite group $G$ let $\Theta(G)$ be the conjugacy class of all $p$-Sylow subgroups of $G$ (cf. \cite[Definition 22.9.1]{FA}). It follows from \cite[Proposition 22.9.2]{FA} that $\Theta$ is a  \cmmnt{consistent} Sylowian map.
\item [(c)] Let $\C$ be a Melnikov formation, i.e. a class of finite groups closed under quotients, normal subgroups and extensions. For a profinite group $G$ let $\Theta(G) = \{G^\C\}$, where $$G^\C = \bigcap_{N\normal G, \ G/N\in\C}N.$$ Thus, $G^\C$ is the normal subgroup of $G$ such that $G/G^\C$ is the maximal pro-$\C$ quotient of $G$ (\cite[Definition 17.3.2]{FA}). By \cite[Lemma 3.4.1]{RZ}(b) and (d), $\Theta$ is a Sylowian map.  \cmmnt{If $\C$ is a full formation, then by \cite[Lemma 3.4.1]{RZ}(c), $\Theta$ is consistent.}
\end{enumerate}
\end{ex}

\begin{defu}\label{def; C of Theta}
Let $\Theta$ be a Sylowian map. Define
$$\C(\Theta)=\{U\;|\: U\in\Theta(G) \text{, where }G \text{ is a finite group}\}.$$ 
Clearly, by Definition \ref{def; Sylowian map}(b),
$\C(\Theta)=\{G\;|\: G \text{ is a finite group such that } \Theta(G)=\{G\}\}$.
\end{defu}

\begin{ex}
Let us describe $\C(\Theta)$ explicitly  for each of the Sylowian maps in Example \ref{examples of Sylowian maps}:
\begin{enumerate}
    \item [(a)] $\C(\Theta)$ consists of the trivial group (resp. of all finite groups).
    \item [(b)] $\C(\Theta)$ consists of all $p$-groups.
    \item [(c)] $\C(\Theta)$ consists of the class of finite groups that have no nontrivial quotient in $\C$.
\end{enumerate}
\end{ex}

By \cite[Proposition 1.14]{FH2},
$\C(\Theta)$ is a \emph{quasi-formation} (cf. \cite[Definition 3.1]{FH1}). In particular, it is closed under taking quotients. In the sequel we will only use the latter property which is necessary for the consideration of pro-$\mathcal{C}(\Theta)$ groups (cf. \cite[p. 19]{RZ}).

The following lemma is taken from the first author's dissertation (\cite{F}):

\begin{lem}\label{lem; ring cover1}
Let $F/E$ be a finite Galois extension and let $\psi$ be a place of $F$, unramified over $E$, with valuation ring $O$. Write $\bar{x}$ for $\psi(x)$ for every $x\in O$. Then:
\begin{enumerate}
\item [(a)] For every $\bar{\sigma}\in\Gal(\overline{E})$ there is a unique $\sigma=\psi^*(\bar{\sigma})\in\Gal(F/E)$ such that for all $x\in O$: 
\begin{equation}\label{equat}
\overline{\sigma(x)}=\bar{\sigma}(\bar{x}).
\end{equation}
Moreover, $\psi^*\colon\Gal(\overline{E})\to\Gal(F/E)$ is a group homomorphism.
\item [(b)] We have $[F:E]=[\overline{F}:\overline{E}]$ if and only if $\psi^*$ is surjective.
\item [(c)] Let $E'$ be an intermediate field of $F/E$. Then $$\psi^*(\Gal(\overline{E'}))=\Gal(F/E')\cap\psi^*(\Gal(\overline{E})).$$
\end{enumerate}
\end{lem}
\begin{proof}
(a) Let $D_\psi$ (resp. $I_\psi$) be the decomposition (resp. inertia) group of $\psi$ in $F/E$. The assumption that $\psi$ is unramified over $E$ means, by definition, that the extension $\overline{F}/\overline{E}$ is separable and $I_\psi$ is trivial (\cite[p. 25]{FA}). By \cite[Lemma 6.1.1(a)]{FA}, $\overline{F}/\overline{E}$ is normal and therefore Galois and there exists a short exact sequence 
\begin{equation*}
1\to I_\psi\to D_\psi\to\Gal(\overline{F}/\overline{E})\to 1.
\end{equation*}

Moreover, the isomorphism $D_\psi\to\Gal(\overline{F}/\overline{E})$ is given by $\tau\mapsto\bar{\tau}$, where $\bar{\tau}(\bar{x})=\overline{\tau(x)}$ for all $x\in O$. Let $\Psi\colon \Gal(\overline{F}/\overline{E})\to D_\psi$ be its inverse. Then for all $\bar{\sigma}\in\Gal(\overline{F}/\overline{E})$ and $x\in O$: $\bar{\sigma}(\bar{x}) = \overline{\Psi(\bar{\sigma})(x)}$. Finally, let $\psi^*=\Psi\circ\res\colon\Gal(\overline{E})\to D_\psi$, where $\res\colon \Gal(\overline{E})\to\Gal(\overline{F}/\overline{E})$ is the restriction map. Then for all $\bar{\sigma}\in\Gal(\overline{E})$ and $x\in O$ we have  $$\overline{\psi^*(\bar{\sigma})(x)}=\overline{\Psi(\res(\bar{\sigma}))(x)}=\res(\bar{\sigma})(\bar{x})= \bar{\sigma}(\bar{x}).$$

To show the uniqueness of $\psi^*(\bar{\sigma})$ for $\bar{\sigma}\in\Gal(\overline{E})$, notice that if $\tau\in\Gal(F/E)$ satisfies (\ref{equat}), i.e. $\overline{\tau(x)}=\bar{\sigma}(\bar{x})$ for all $x\in O$, then $\tau\in D_\psi$. As $D_\psi\cong\Gal(\overline{F}/\overline{E})$, $\psi^*(\bar{\sigma})$ is unique.

(b) We have seen in (a) that the image of $\psi^*$ is the decomposition group $D_\psi$ which has order $[\overline{F}:\overline{E}]$. Thus, the image of $\psi^*$ is $\Gal(F/E)$ if and only if $[F:E]=[\overline{F}:\overline{E}]$.

(c) Apply (a) to $E'$ instead of $E$. By the uniqueness, the resulting homomorphism $\Gal(\overline{E'})\to\Gal(F/E')$ is the restriction of $\psi^*$ to $\Gal(\overline{E'})$. Thus, $\psi^*(\Gal(\overline{E'}))\subseteq \Gal(F/E')$. Conversely, let $\bar{\sigma}\in\Gal(\overline{E})$ with $\sigma=\psi^*(\bar{\sigma})\in\Gal(F/E')$. Then $\bar{\sigma}(\bar{x})=\overline{\sigma(x)}=\bar{x}$ for every $x\in O\cap E'$. Thus, $\bar{\sigma}\in\Gal(\overline{E'})$.
\end{proof}

In contrast to the standard definition of Hilbertianity (cf. \cite[p. 219]{FA}), our definition of $\Theta$-Hilbertianity uses the terminology of places and Galois groups. It is justified by the following consideration: By \cite[Lemma 12.1.6]{FA}
(with the slight adjustment that the extension
$K(\textbf{T},y)/K(\textbf{T})$
in their proof should be required to be Galois),
a field $K$ is Hilbertian
if and only if for every finite Galois extension $F/K(t)$ 
there are infinitely many $a \in K$ for which  there exists a $K$-place $\varphi\colon F \to \overline{F}\cup\{\infty\}$ with $\varphi(t) = a$ and $[\overline{F} : K]= [F:K(t)]$.


\begin{defu}\label{def; Hilbertian}
Let $K$ be a field and let $\Theta$ be a Sylowian map. Let $F/K(t)$ be a finite Galois extension and let $F_\Theta\in\Theta(F/K(t))$. We denote by $H_{K,\Theta}(F)$ the set of all $a \in K$ such that there exists a $K$-place $\varphi\colon F \to \overline{F}\cup\{\infty\}$ with $\varphi(t) = a$ and $[\overline{F} : \overline{F_\Theta}]= [F:F_\Theta]$.
Notice that $H_{K,\Theta}(F)$ does not depend on the choice of $F_\Theta$, which is unique up to conjugation in $\Gal(F/K(t))$.

We say that
$K$ is \textbf{$\Theta$-Hilbertian}
if $H_{K,\Theta}(F)$ is infinite for every finite Galois extension $F/K(t)$ satisfying $\Gal(F/K(t))\in\C(\Theta)$, i.e. $\Theta(F/K(t))=\{K(t)\}$.
We say that
$K$ is \textbf{strongly $\Theta$-Hilbertian}
if $H_{K,\Theta}(F)$ is infinite for every finite Galois extension $F/K(t)$.

In particular,
if $p$ is a prime number,
we say that
\textbf{$K$ is $p$-Hilbertian} (resp., \textbf{strongly $p$-Hilbertian})
if $K$ is $\Theta$-Hilbertian (resp., strongly $\Theta$-Hilbertian),
where $\Theta(G)$ is the conjugacy class of $p$-Sylow subgroups of $G$,
for every profinite group $G$
(Example \ref{examples of Sylowian maps}(b)).
\end{defu}

Thus, a field $K$ is Hilbertian if and only if it is $\Theta$-Hilbertian,
where $\Theta$ is the identity Sylowian map 
(Example \ref{examples of Sylowian maps}(a)).

\subsection{Embedding problems and the field crossing argument}\label{sec: absolute Galois group}


\begin{defu}\label{def; emb}
Let $A,B$ and $G$ be profinite groups. 
An \textbf{embedding problem for $G$} is a pair
\begin{equation}\label{eq; embedding problem}(\varphi\colon G\to A, \alpha\colon B\to A) \end{equation}
in which $\varphi$ and $\alpha$ are epimorphisms. We call (\ref{eq; embedding problem}) \textbf{finite} if $B$ is finite and \textbf{split} if there exists a homomorphism $\alpha'\colon A\to B$ with $\alpha\circ\alpha'=\id_A$. A \textbf{weak solution} to (\ref{eq; embedding problem}) is a homomorphism $\gamma\colon G\to B$ such that $\alpha\circ\gamma=\varphi$. A weak solution $\gamma$ to (\ref{eq; embedding problem}) is a \textbf{solution} if $\gamma$ is surjective.
\end{defu}

\begin{defu}
Let $L/K$ be a finite Galois extension of fields and identify $\Gal(L(t)/K(t))$ with $A=\Gal(L/K)$ via restriction. Two special cases are of interest:
\begin{enumerate}
\item [(a)] Let $G=\Gal(K(t))$ and $\varphi = \res_{K(t)_s/L}$. We refer to (\ref{eq; embedding problem}) as a \textbf{constant embedding problem over $K(t)$} and call a solution $\gamma\colon G\to A$ to (\ref{eq; embedding problem}) \textbf{regular} if $F/L$ is a regular field extension where $F$ is the fixed field of $\ker(\gamma)$ in $K(t)_s$ (cf. \cite[Section 4.4]{Patching}).
\item [(b)] Let $G=\Gal(K)$ and $\varphi = \res_{K_s/L}$. We say that $K$ is \textbf{$\omega$-free}, if every finite embedding problem for $G$ has a solution (cf. \cite[Section 5.10]{Patching}).
\end{enumerate}



\end{defu}

A main ingredient in the proof of our main results is the field crossing argument (cf. \cite[Section 24.1]{FA}). It provides a connection between homomorphisms of Galois groups and places of fields. Our version is more general than the one mentioned (we do not consider only PAC fields) and tailored to our needs.

\begin{prop}\label{field-crossing(a)}
Let $E/K$ be a regular extension of fields and let $F/E$ and $M/K$ be finite Galois extensions. Assume that $M$ contains the algebraic closure $L$ of $K$ in $F$. 
\begin{enumerate}
\item [(a)] Let $\gamma \colon \Gal(K) \to \Gal(F/E)$ be a homomorphism such that $\res_{F/L}\circ\gamma=\res_{K_s/L}$ and $\Gal(M)\leq\ker(\gamma)$. 
Then there exists a unique homomorphism $\gamma'$ such that the following diagram commutes:
\begin{equation}\label{groups crosing diagram}
\xymatrix{
\Gal(K)
\ar@/_{3pc}/[dd]_{\gamma}\ar[dr]^{\res_{K_s/M}} \ar[d]^{\gamma'} \\
\Gal(FM/E) \ar[d]^{\res_{FM/F}} \ar[r]^{\res_{FM/M}} & \Gal(M/K)\rlap{$\cong \Gal(EM/E)$}\ar[d]^{\res_{M/L}} \\
\Gal(F/E) \ar[r]^{\res_{F/L}}&\Gal(L/K)\rlap{$\cong \Gal(EL/E)$} \\
}
\end{equation}
Moreover,
let $D$ be the fixed field of $\Delta = \gamma'(\Gal(K))$ in $FM$. Then
\begin{itemize}
\item[(a1)]
the square in the above diagram is cartesian,
\item[(a2)]
$D$ is a regular extension of $K$,
\item[(a3)]
$E \subseteq D \subseteq FM$ and $DM = FM$,
and
\item[(a4)]
$\gamma(\Gal(K))=\Gal(F/F\cap D)$.
\end{itemize}


\item [(b)] Let $E\subseteq D'\subseteq FM$ be an intermediate field, regular over $K$, such that $D'M = FM$. Suppose there exists a $K$-place $\varphi'\colon D'\to K\cup\{\infty\}$, unramified over $E$. Extend $\varphi'$ to an $M$-place $\varphi$ of $FM$ and let $\psi$ be the restriction of $\varphi$ to $F$. Then the homomorphism $\psi^*\colon\Gal(K)\to\Gal(F/E)$ defined in Lemma \ref{lem; ring cover1}(a) (notice that $\overline{E}=K$) satisfies 
$\psi^*(\Gal(K))=\Gal({F/F\cap D'})$.
\end{enumerate}
\end{prop}

\begin{proof}
\begin{enumerate}
\item [(a)] Since $F$ and $M$ are linearly disjoint over $L$, by \cite[Lemma 2.5.3]{FA}, $F$ and $(EL)M = EM$ are linearly disjoint over $EL$. In particular, $F \cap EM = EL$. Clearly, $F(EM) = FM$. Hence, by Galois theory, the square in (\ref{groups crosing diagram}) is cartesian. Since $\res_{F/L}\circ\gamma=\res_{K_s/L}$, by the universal property of cartesian squares (\cite[Proposition 22.2.1(b)]{FA}), there exists a unique homomorphism $\gamma'$ such that (\ref{groups crosing diagram}) commutes.

Let $h\colon\Delta\to\Gal(M/K)$ be the restriction of $\res_{FM/M}$ to $\Delta$. We claim that $h$ is an isomorphism. Indeed, $h\circ\gamma'=\res_{K_s/M}$ and $\res_{K_s/M}$ is surjective, hence, $h$ is surjective. Now, let $\sigma\in\Delta$ with $h(\sigma)=1$, i.e., $\res_{FM/M}(\sigma)=1$. There exists $\tau\in\Gal(K)$ such that $\gamma'(\tau)=\sigma$. Thus, $\res_{K_s/M}(\tau)=1$. Hence, $\tau\in\Gal(M)\leq\ker(\gamma)$ and therefore $\gamma(\tau) = 1$. Thus, $\res_{FM/F}(\sigma) = \res_{FM/F}(\gamma'(\tau)) =  \gamma(\tau) = 1$. Conclude that $\sigma=1$. Hence,
$h$ is injective and therefore an isomorphism.

It follows that $DM = FM$ and $M\cap D=K$. In particular, $D$ and $M$ are linearly disjoint over $K$. Also $DM$ and the algebraic closure $\Mgal$ of $M$ are linearly disjoint over $M$, since $DM = FM$ is regular over $M$. By \cite[Lemma 2.5.3]{FA} $D$ and $\Mgal$ (which is also an algebraic closure of $K$) are linearly disjoint over $K$. Hence, $D/K$ is regular. Moreover, the restriction of $\res_{FM/F}$ to $\Delta$ is onto $\gamma(\Gal(K))$. Therefore, $F\cap D$ is the fixed field of $\gamma(\Gal(K))$ in $F$.

\item [(b)] It holds, with respect to $\varphi$, that $\overline{D'} = K$ and $\overline{M} = M$, hence
$$[D'M:D']\leq[M:K]\leq[\overline{D'M}:\overline{D'}]\leq[D'M:D'].$$ Thus, these inequalities are in fact equalities. Hence, $\varphi$ is unramified over $E$ and therefore the same holds for $\psi$. 

We wish to apply (a) to $\gamma=\psi^*$. First, since $\psi$ is the restriction of an $M$-place, $\Gal(M)\leq\ker(\psi^*)$. Second, $\res_{F/L}\circ\psi^*= \res_{K_s/L}$. Indeed, let $\bar{\sigma}\in\Gal(K)$ and $x\in L$. Denote $\sigma=\psi^*(\bar{\sigma})$. Since $\psi$ is an $L$-place, $\bar{x}=x$ and $\overline{\sigma(x)}=\sigma(x)$. It follows from \eqref{equat} that $\sigma(x)=\bar{\sigma}(x)$. This shows that $(\res_{F/L}\circ\psi^*)(\bar{\sigma})= \res_{K_s/L}(\bar{\sigma})$. 

Now, by the uniqueness of $\gamma'$ in (a), necessarily $\gamma'=\varphi^*$ and we obtain from (a) an intermediate field $E\subseteq D\subseteq FM$ such that $\Gal(FM/D)=\gamma'(\Gal(K))$. By Lemma \ref{lem; ring cover1}(c) (where we take $E'$ to be $D'$ and $F$ to be $FM$ and use  $\overline{D'}=K$), $\Gal(FM/D)\leq\Gal(FM/D')$ or, equivalently, $D'\subseteq D$. As $[DM:D]=[M:K]$, we have $D'=D$. Thus, by (a4),
$\psi^*(\Gal(K))=\gamma(\Gal(K))=\Gal(F/F\cap D')$.
\end{enumerate}
\end{proof}

\section{Main results}\label{Main Results}

The question that we address in this work, namely, whether $\Theta$-Hilbertianity and strong $\Theta$-Hilbertianity coincide, generalizes a question asked by Moshe Jarden in an unpublished manuscript \cite{p-Hilbert}, where he defines $p$-Hilbertianity and strong $p$-Hilbertianity (Definition~\ref{def; Hilbertian}) for fields with pro-$p$ absolute Galois group.

\subsection{PAC fields}\label{section PAC}

Recall that a field $K$ is called \textbf{pseudo algebraically closed (PAC)} if every absolutely irreducible variety defined over $K$ has a $K$-rational point. 
Examples of PAC fields include, among others,
infinite models of the theory of finite fields (\cite[Corollary 20.10.5]{FA}),
infinite algebraic extensions of finite fields (\cite[Corollary 11.2.4]{FA}),
and the field $R(\sqrt{-1})$,
where $R$ is the field of totally real algebraic numbers
(\cite[Main Theorem]{GPR}).
It seems that PAC fields appear for the first time in \cite{Ax} (without an explicit name).

Since their appearance, PAC fields have been extensively studied and shown to have many more nice properties, e.g., they are  $\omega$-free if and only if they are Hilbertian (\cite[Theorem 5.10.3]{Patching}). Thus, it may come as no surprise that for PAC fields the answer to Jarden's question is positive:

\begin{theorem}\label{PAC p Hilbertianity}
Let $\Theta$ be a Sylowian map and let $K$ be a PAC $\Theta$-Hilbertian field with $\Gal(K)$ pro-$\C(\Theta)$.
Then $K$ is strongly $\Theta$-Hilbertian.
\end{theorem}

\begin{proof}
Let $F/K(t)$ be a finite Galois extension.
We need to show that $H_{K,\Theta}(F)$ is infinite. 

Let $L$ be the algebraic closure of $K$ in $F$
and denote $A=\Gal(L/K)$.
The restriction maps
$\res_{F/L}\colon \Gal(F/K(t))\to A$
and
$\res_{K_s/L}\colon \Gal(K)\to A$
are both onto $A$.
Let $F_\Theta\in\Theta(F/K(t))$
and let $B=\Gal(F/F_\Theta)$.
As $\Gal(K)$ is pro-$\C(\Theta)$,
its image $A$ is in $\C(\Theta)$,
i.e. $\Theta(A)=\{A\}$.
By Definition~\ref{def; Sylowian map}(a),
$\res_{F/L}(B)\in \Theta(A)$.
Hence, $\res_{F/L}(B)=A$.
Thus, $F_\Theta$ and $L$ are linearly disjoint over $K$,
and hence $F_\Theta$ is regular over $K$.
Let $\alpha\colon B\to A$ be the restriction of $\res_{F/L}$ to $B$. 

By \cite[Proposition 3.3]{FLP} the constant embedding problem $(\res_{K(t)_s/L}\colon \Gal(K(t)) \to  A, \alpha\colon B\to A)$ has a solution. Thus, there exists a Galois extension $N$ of $K(t)$ containing $L$ and an isomorphism $\theta\colon\Gal(N/K(t))\to B$ such that $\alpha\circ\theta=\res_{N/L}$.

As $K$ is $\Theta$-Hilbertian and $B\in\C(\Theta)$, there exists a $K$-place $\rho\colon N\to K_s\cup\{\infty\}$ such that $\overline{K(t)}=K$ and $[\overline{N}:K]=[N:K(t)]$. It follows from Lemma \ref{lem; ring cover1}(a) and (b) that $\rho^*\colon\Gal(K)\to\Gal(N/K(t))$ is surjective and satisfies $\res_{N/L}\circ\rho^*=\res_{K_s/L}$. Then $\gamma=\theta\circ\rho^*\colon\Gal(K)\to B$ is an epimorphism such that $\alpha\circ\gamma=\res_{K_s/L}$. 

Let $M$ be the fixed field of $\ker(\gamma)$ in $K_s$. Thus, $M/K$ is a finite Galois extension. Since $\alpha\circ\gamma=\res_{K_s/L}$,\ $M$ contains $L$. By Proposition~\ref{field-crossing(a)}(a) there is a regular extension $D/K$ such that $F_\Theta\subseteq D$ and $DM=FM$. Furthermore, $\Gal(F/F_\Theta)=B=\gamma(\Gal(K))=\Gal(F/F\cap D)$. Thus, $F\cap D=F_\Theta$.

As $K$ is PAC, there exist infinitely many $K$-places $\varphi'\colon D\to K\cup\{\infty\}$.
Extend such $\varphi'$ to an $M$-place of $DM=FM$ and let $\varphi$ be its restriction to $F$.
Since only finitely $K$-places of $FM$ are ramified over $K(t)$ (\cite[p.~59]{FA}), we may assume that 
$\varphi'$ and therefore also $\varphi$, is unramified over $K(t)$. By Proposition~\ref{field-crossing(a)}(b), $\varphi^*\colon\Gal(K)\to\Gal(F/F_\Theta)$ is an epimorphism. Thus, by Lemma \ref{lem; ring cover1}(b) with respect to $\varphi$, \ $[F:F_\Theta]=[\overline{F}:\overline{F_\Theta}]$. 
\end{proof} 

\subsection{PRC fields}

In this section we generalize Theorem \ref{PAC p Hilbertianity} to PRC fields. Let us introduce the notation and recall the necessary facts needed to this end (mostly taken from \cite{FHV}):
An extension $E/K$ of fields is \textbf{totally real} if every ordering on $K$ extends to $E$.
A field $K$ is called \textbf{pseudo real closed (PRC)} if every absolutely irreducible variety defined over $K$ has a $K$-rational point, provided it has a non-singular point over every real closure of $K$.
Equivalently, for every finitely generated totally real and regular extension $E/K$ there exist infinitely many $K$-places $\varphi\colon E\to K\cup\{\infty\}$.

Let $K$ be a field. By \cite[\S 6]{Prestel}, the set of orderings $X(K)$ of $K$ is a compact, Hausdorff and totally disconnected topological space with the topology given by a subbase consisting of sets of the form $H(c) = \{P\in X(K)\,|\, c\in P\}$ for $c\in K^\times$. Here $P$ denotes the positive elements in an ordering. Each of the sets $H(c)$ is open and closed (\textbf{clopen}).

A subset $I$ of a group $G$ is called a \textbf{conjugacy domain} if $I=\bigcup_{\sigma\in G} I^\sigma$. 

A field $K$ is called \textbf{formally real} if $K$ admits at least one ordering. A formally real field is of characteristic~$0$. In what follows $F/E$ is a Galois extension of fields with $F$ not formally real. An involution $\varepsilon\in\Gal(F/E)$ (that is, an element of order~$2$) is \textbf{real} if its fixed field $F(\varepsilon)$ in $F$ is formally real. Denote the set of real involutions of $\Gal(F/E)$ by $I(F/E)$. This is a closed subset of $\Gal(F/E)$. By Artin-Schreier theory, every involution $\varepsilon$ of $\Gal(E)$ is real and self-centralizing, that is, $\{\sigma \in \Gal(E)\; |\; \varepsilon \sigma = \sigma \varepsilon\} = \{1,\varepsilon\}$.

Suppose $E/K$ is totally real and let $P\in X(K)$. Denote the set of involutions $\varepsilon\in I(F/E)$ for which $P$ extends to an ordering of $F(\varepsilon)$ by $I_P(F/E)$. For $X\subseteq X(K)$ let $I_X(F/E)=\bigcup_{P\in X} I_P(F/E)$. If $F$ is the algebraic closure of $E$, we write $I_P(E)$ for $I_P(F/E)$, etc.

\begin{rem}\label{rem; orderings}
(a) Suppose $E=K$. Then $I_P(F/K)$ is a conjugacy class in $\Gal(F/K)$ (\cite[Remark 1.8(a)]{FHV}). In the general case, $I_P(F/E)$ is a conjugacy domain in $\Gal(F/E)$. In fact, $I_P(F/E)=\bigcup_{\substack{Q\in X(E) \\ Q\supseteq P}} I_Q(F/E)$.

(b) Let $I\subseteq I(F/E)$ be a conjugacy domain. From \cite[Theorems 4.1 and 4.9]{ELW} it follows that for every $\epsilon\in I(F/E)$ the set $X_\epsilon=\{P\in X(K)\;|\;P \text{ extends to } F(\epsilon)\}$ is clopen in $X(K)$. Hence, $$\{P\in X(K)\;|\; I_P(F/E)=I\}=\bigcap_{\epsilon\in I}X_\epsilon\smallsetminus \bigcup_{\epsilon\in I(F/E)\smallsetminus I}X_\epsilon$$ is clopen in $X(K)$.

(c) Suppose $E=K(t)$. Then there exists a polynomial $g\in K[t]$ such that for every $a\in K$ with $g(a)\neq 0$ the map $t\mapsto a$ extends to a $K$-place $\varphi\colon F\to K_s\cup\{\infty\}$ with $\varphi^*(I_P(K))\subseteq I_P(F/E)$ for every $P\in X(K)$ (\cite[Remark 6.2(b)]{FHV}).
\end{rem}

Finally, let us recall the PRC counterpart of \cite[Proposition 3.3]{FLP}:

\begin{lem}{\cite[Theorem 5.2]{FHV}}\label{lem; 5.2}
Let $K$ be a PRC field.
Let $L/K$ be a finite Galois extension with $L$ not formally real and let $\pi\colon H\to\Gal(L/K)$ be an epimorphism of finite groups. Let $X_1, \ldots, X_m$ be a partition of $X(K)$ into disjoint clopen sets. For each $1\leq j\leq m$ let $I_j\subseteq H$ be a conjugacy domain of involutions such that $\pi(I_j)=I_{X_j}(L/K)$. Then there exists a regular extension $F$ of $L$, Galois over $K(t)$, and an isomorphism $h_1\colon\Gal(F/K(t))\to H$ that maps $I_{X_j}(F/K(t))$ onto $I_j$ that makes the following diagram commutative:

$$
\xymatrix{
\Gal(F/K(t)) \ar[dr]_{\res_{F/L}} \ar[rr]^{h_1} && H \ar[dl]^{\pi} \\
& \Gal(L/K) \\
}
$$
\end{lem}

\begin{lem}\label{kill centralizers}
Let $G=\varprojlim_{i\in I} G_i$ be an inverse limit of finite groups with canonical projections $\pi_i\colon G\to G_i$ and connecting epimorphisms $\pi_{ji}\colon G_j\to G_i$ for $j\geq i$. Let $\mathcal{I}$ be the set of involutions of $G$ and suppose that every $\delta\in\mathcal{I}$ is self-centralizing. Let $i\in I$ be such that $1\notin\pi_i(\mathcal{I})$. 
Then there exists $j\geq i$ such that for every $\delta\in \mathcal{I}$ the image of the centralizer of $\pi_j(\delta)$ in $G_i$ under $\pi_{ji}$ is  $\{1,\pi_i(\delta)\}$.
\end{lem}

\begin{proof}
For every $j\in I$ with $j\geq i$ denote
\begin{equation*}
D_j = \{(\delta,\sigma) \in G\times G\; |\; \delta^2 = 1, \pi_i(\delta)\neq 1, \pi_i(\sigma)\notin\{1, \pi_i(\delta)\}, \pi_j(\sigma\delta)=\pi_j(\delta\sigma)\}.
\end{equation*}
This is a closed subset of $G \times G$ and if $j'\geq j$, then $D_{j'}\subseteq D_j$. 

Let 
$\delta\in \mathcal{I}$. Suppose there exists $\tau\in G_j$ that centralizes $\pi_j(\delta)$ but $\pi_{ji}(\tau)\notin\{1, \pi_i(\delta)\}$. Then $(\delta,\sigma)\in D_j$ for every $\sigma\in G$ with $\pi_j(\sigma)=\tau$. In particular, $D_j\neq\emptyset$. Thus, it suffices to show that there exists $j\geq i$ with $D_j = \emptyset$.

Let $(\delta,\sigma)\in \bigcap_{j\geq i} D_j$. As $\delta^2 = 1$ and $\pi_i(\delta)\neq 1$, we have $\delta\in \mathcal{I}$. Furthermore, $\pi_j(\sigma\delta)=\pi_j(\delta\sigma)$ for every $j\geq i$. Thus, $\sigma\delta=\delta\sigma$ and therefore $\sigma$ centralizes $\delta$. By assumption, $\sigma\in\{1,\delta\}$, contrary to the definition of $D_j$. Thus, $\bigcap_{j\geq i} D_j=\emptyset$. Since $G\times G$ is compact, there exist $i\leq j_1,\ldots,j_n\in I$ such that $\bigcap_{1\leq k\leq n} D_{j_k}=\emptyset$. Let $j\in I$ be such that $j\geq\max\{j_1,\ldots,j_n\}$. Then $D_j\subseteq \bigcap_{1\leq k\leq n}D_{j_k}$ and therefore $D_j=\emptyset$.
\end{proof}

\begin{cor}\label{kill centralizers of real involutions}
Let $L/K$ be a finite Galois extension. 
Then there exists a finite Galois extension $L'$ of $K$ containing $L$ such that the restriction to $L$ of the centralizer of every real involution $\varepsilon \in I(L'/K)$ is $\{1,\res_{L'/L}\varepsilon\}\leq\Gal(L/K)$.
\end{cor}

\begin{proof}
By Artin-Schreier theory, there are involutions in $G=\Gal(K)$ only if $\textnormal{char}(K)=0$. Furthermore, every involution is self-centralizing. Adjoin $\sqrt{-1}$ to $L$ to assume that $L$ is not formally real and therefore no involution of $G$ restricts to $1\in \Gal(L/K)$. The assertion now follows from Lemma \ref{kill centralizers}.
\end{proof}

\begin{lem}\label{epimorphism onto Sylow subgroup}
Let $\Theta$ be a Sylowian map and let $K$ be a PRC $\Theta$-Hilbertian field such that $\Gal(K)$ is pro-$\C(\Theta)$. Let $F/K(t)$ be a finite Galois extension such that the algebraic closure $L$ of $K$ in $F$ is not formally real. Let $E\in\Theta(F/K(t))$  such that the extension $E/K$ is totally real. Then there exists an epimorphism $\gamma\colon \Gal(K)\to\Gal(F/E)$ such that $\res_{F/L}\circ\gamma= \res_{K_s/L}$ and $\gamma(I_P(K))\subseteq I_P(F/E)$ for every $P\in X(K)$.
\end{lem}

\begin{proof}
By assumption, $E/K$ is totally real and therefore the restriction map $X(E) \to X(K)$ is surjective.
Let $\rho \colon X(K) \to X(E)$ be a continuous section (\cite[Proposition 8.2]{HJ-prc}) of this map and let $P\in X(K)$. By Remark \ref{rem; orderings}(b), if $P'\in X(K)$ is sufficiently close to $P$ then $I_{\rho(P')}(F/E) = I_{\rho(P)}(F/E)$. Thus, we may construct a partition of $X(K)$ into disjoint clopen subsets $X_1,\ldots, X_m$, and, for each $1\leq j\leq m$, a conjugacy class of involutions $I_j\subseteq I(F/E)$ such that every $P\in X_j$ extends to $F(\varepsilon)$ for every $\varepsilon\in I_j$. In particular, $P$ extends to $F(\varepsilon)\cap L$ and therefore $I_P(L/K)\cap\res_{F/L}(I_j)\neq\emptyset$. By Remark \ref{rem; orderings}(a), $I_P(L/K)$ is a conjugacy class. By assumption, $\Gal(K)$ is pro-$\C(\Theta)$ and therefore the restriction of $\res_{F/L}$ to $\Gal(F/E)$ is surjective. It follows that also $\res_{F/L}(I_j)$ is a conjugacy class. Conclude that $I_P(L/K)=\res_{F/L}(I_j)$. Finally, $I_{X_j}(L/K)=\bigcup_{P\in X_j}I_P(L/K)=\res_{F/L}(I_j)$.

By Lemma \ref{lem; 5.2}, there exist a Galois extension $F'$ of $K(t)$ containing $L$ and an isomorphism $h\colon\Gal(F'/K(t)) \to \Gal(F/E)$ such that the following diagram commutes
$$
\xymatrix{
\Gal(F'/K(t)) \ar[dr]_{\res_{F'/L}} \ar[rr]^{h} && \Gal(F/E) \ar[dl]^{\res_{F/L}} \\
& \Gal(L/K) \\
}
$$
and $h(I_{X_j}(F'/K(t)))=I_j$ for each $1\leq j\leq m$.

As $K$ is $\C(\Theta)$-Hilbertian, there exists a $K$-place $\varphi\colon F'\to \overline{F'}\cup\{\infty\}$ with $\overline{K(t)}=K$ and $[\overline{F'}:K]=[F':K(t)]$. We may assume that $\varphi$ is an $L$-place. Then the group homomorphism $\varphi^*\colon\Gal(K)\to\Gal(F'/K(t))$, introduced in Lemma~\ref{lem; ring cover1}, is surjective, $\res_{F'/L}\circ\varphi^*=\res_{K_s/L}$ and, by Remark \ref{rem; orderings}(c), $\varphi^*(I_P(K)) \subseteq I_P(F'/K(t))$ for every $P\in X(K)$. It follows that $\gamma=h\circ\varphi^*\colon\Gal(K)\to \Gal(F/E)$ is an epimorphism such that $\res_{F/L} \circ \gamma = \res_{K_s/L}$ and $\gamma(I_P(K))\subseteq I_P(F/E)$ for every $P\in X(K)$. 
\end{proof}

We can improve the conclusion of the preceding lemma:

\begin{lem}\label{epimorphism onto Sylow subgroup and orderings}
Let $\Theta$ be a Sylowian map and let $K$ be a field such that for every $\Ehat\in\Theta(K(t))$ the extension $\Ehat/K$ is totally real. Under the assumptions of the preceding lemma there exists an epimorphism $\gamma \colon \Gal(K) \to \Gal(F/E)$ such that $\res_{F/L} \circ \gamma = \res_{K_s/L}$ and for every $\delta \in I(K)$ the restriction of the unique ordering on $K_s(\delta)$ to $L(\delta)$ extends to an ordering on $F(\gamma(\delta))$. 
\end{lem}

\begin{proof}
By Corollary~\ref{kill centralizers of real involutions}, there exists a finite Galois extension $L'$ of $K$ containing $L$ such that the restriction to $L$ of the centralizer of every real involution $\varepsilon'\in I(L'/K)$ is $\{1,\res_{L'/L}\varepsilon'\}\leq\Gal(L/K)$. Let $F'=FL'$. Then $L'$ is the algebraic closure of $K$ in $F'$.

By Definition~\ref{def; Sylowian map}(a) there is
$\Ehat\in\Theta(K(t))$ such that the restriction of automorphisms maps $\Gal(\Ehat)$ onto $\Gal(F/E)$. Put $E' = F'\cap \Ehat$. Then the restriction $\res_{F'/F}$ maps $\Gal(F'/E')$ onto $\Gal(F/E)$ and, again by Definition~\ref{def; Sylowian map}(a), $E'\in\Theta(F'/K(t))$. By assumption, the extension $\Ehat/K$ is totally real and therefore so is $E'/K$. We apply the preceding lemma to $F'/L'$ instead of $F/L$: There exists an epimorphism $\gamma'\colon\Gal(K)\to\Gal(F'/E')$ such that $\res_{F'/L'} \circ \gamma' = \res_{K_s/L'}$ and $\gamma'(I_P(K))\subseteq I_P(F'/E')$ for every $P\in X(K)$. Let $\gamma=\res_{F'/F}\circ\gamma'$.  Then $\res_{F/L} \circ \gamma =\res_{K_s/L}$.

Let $\delta \in I(K)$ and let $\varepsilon$ (resp. $\varepsilon'$) be its restriction to $L$ (resp. $L'$). Let $P$ (resp. $P_1,P'_1$) be the restriction to $K$ (resp. $L(\varepsilon), L'(\varepsilon')$) of the unique ordering on $K_s(\delta)$. Then, $\delta\in I_P(K)$ and therefore $\gamma'(\delta)\in I_P(F'/E')$. Hence, $P$ extends to an ordering $Q'$ on $F'(\gamma'(\delta))$. Let $Q$ be the restriction of $Q'$ to $F(\gamma(\delta))$ and let $P_2$ (resp. $P'_2$) be the restriction of $Q$ (resp. $Q'$) to $L(\varepsilon)$ (resp. $L'(\varepsilon')$).

Both $P'_1$ and $P'_2$ extend $P$. By \cite[Proposition 2.1(iii)]{HJ-prc}, there exists $\sigma'\in\Gal(L'/K)$ such that $\sigma' (L'(\varepsilon')) = L'(\varepsilon')$ and $\sigma'(P'_1)=P'_2$. Thus, $(\varepsilon')^{\sigma'}=\varepsilon'$. Let $\sigma = \res_{L'/L}\sigma'$. Then $\sigma(P_1)=P_2$ and $\varepsilon^\sigma = \varepsilon$.

By the first paragraph of this proof, $\sigma \in \{1, \varepsilon\}$. Hence $P_1=\sigma(P_1)=P_2$. Thus, $P_1$ extends to $Q$.
\end{proof}

\begin{lem}\label{restriction of square}
Let $F_1$ and $F_2$ be two Galois extensions of a field $K$. Let $F_0=F_1\cap F_2$ and $F =F_1F_2$. Let $H\leq\Gal(F/K)$ and assume that the restriction to $H$ of the epimorphism $\res_{F/F_0}$ is injective. Let $F'$ be the fixed field of $H$ in $F$ and let $F_1', F_2', F_0'$ be the fixed fields of the images of $H$ in $F_1, F_2, F_0$, respectively. Then $F_1'$ and $F_2'$ are linearly disjoint over $F_0'$ and $F_1'F_2' = F'$.
\end{lem}

\begin{proof}
By assumption, the restriction map $\Gal(F/F') \to \Gal(F_0/F_0')$ is an isomorphism. Hence, so are the restrictions $\Gal(F_1/F_1')\to \Gal(F_0/F_0')$ and $\Gal(F/F')\to\Gal(F_2/F_2')$. By Galois theory, $F_1'\cap F_0=F_0'$ and $F_1'F_0 = F_1$. Since $F_0/F_0'$ is a Galois extension, $F_1'$ and $F_0$ are linearly disjoint over $F_0'$ (\cite[first paragraph on p. 35]{FA}). Similarly, $F_1$ and $F_2$ are linearly disjoint over $F_0$. By the tower property of linear disjointness (\cite[Lemma 2.5.3]{FA}), $F_1'$ and $F_2$ are linearly disjoint over $F_0'$. In particular, $F_1'$ and $F_2'$ are linearly disjoint over $F_0'$.

Now, $F_1'F_2 = F_1'(F_0F_2) = (F_1'F_0)F_2 = F_1F_2 = F$. Again, by \cite[first paragraph on p. 35]{FA}, $F'$ and $F_2$ are linearly disjoint over $F_2'$. Thus, by the tower property, $F'$ and $(F_1'F_2')F_2$ are linearly disjoint over $F_1'F_2'$. By \cite[Corollary 2.5.2]{FA}, $[F':F_1'F_2']=[F:(F_1'F_2')F_2]$. But $(F_1'F_2')F_2=F_1'F_2=F$, so $F_1'F_2'=F'$.
\end{proof}

\begin{theorem}\label{PRC p Hilbertianity}
Let $\Theta$ be a Sylowian map and let $K$ be a PRC $\C(\Theta)$-Hilbertian field with $\Gal(K)$ pro-$\C(\Theta)$. Suppose that for every $\Ehat\in\Theta(K(t))$ the extension $\Ehat/K$ is totally real. Then $K$ is strongly $\C(\Theta)$-Hilbertian.
\end{theorem}

\begin{proof}
If $K$ admits no orderings, then $K$ is PAC and the assertion follows from Theorem~\ref{PAC p Hilbertianity}. Suppose otherwise. Let $F/K(t)$ be a finite Galois extension and let $E\in\Theta(F/K(t))$. As in the proof of Lemma \ref{epimorphism onto Sylow subgroup and orderings}, the extension $E/K$ is totally real. Let $L$ be the algebraic closure of $K$ in $F$. By \cite[Lemma 2.2]{FH2}, we may replace $F$ by a field containing $F$ and thus assume that $\sqrt{-1}\in L$.

By Lemma~\ref{epimorphism onto Sylow subgroup and orderings}, there exists an epimorphism $\gamma\colon\Gal(K)\to \Gal(F/E)$ such that \linebreak $\res_{F/L}\circ\gamma=\res_{K_s/L}$ and for every $\delta\in I(K)$ the restriction of the unique ordering on $K_s(\delta)$ to $L(\delta)$ extends to an ordering on $F(\gamma(\delta))$.

Let $M$ be the fixed field of $\ker(\gamma)$ in $K_s$.
This is a finite Galois extension of $K$.
As $\res_{F/L}\circ\gamma=\res_{K_s/L}$, we have
$\ker(\gamma) \le \ker(\res_{K_s/L})$, and hence $L \subseteq M$, and $\res_{F/L}$ is surjective, and hence $E\cap L = K$.
Thus $K$ is algebraically closed in $E$.
Since $K(t)/K$ and $F/E$ are separable and $K \subseteq E \subseteq F$, also $E/K$ is separable.
Therefore $E/K$ is regular.

By Proposition~\ref{field-crossing(a)}(a) there is a commutative diagram~\eqref{groups crosing diagram} such that the following holds: the fixed field $D$ of $\gamma'(\Gal(K))$ in $FM$ is a regular extension of $K, DM=FM$ and $F\cap D= E$.

We claim that $D/K$ is totally real. Indeed, let $P$ be an ordering on $K$ and let $\delta \in I(K)$ such that $(K_s(\delta), \hat{P})$ is a real closure of $(K,P)$. Let $P'$ be the restriction of $\hat{P}$ to $M(\delta)$. Then the restriction of $P'$ to $L(\delta)$ extends to an ordering $Q$ on $F(\gamma(\delta))$. Clearly, $F(\gamma(\delta))\cap M(\delta) = L(\delta)$. By Lemma \ref{restriction of square}, $F(\gamma(\delta))$ and $M(\delta)$ are linearly disjoint over $L(\delta)$ and $F(\gamma(\delta))M(\delta) = FM(\gamma'(\delta))$. By \cite[p. 241]{Jarden-e-fold}, there exists an ordering on $FM(\gamma'(\delta))$ that extends both $P'$ and $Q$. In particular, $P$ extends to $D\subseteq FM(\gamma(\delta))$.

As $K$ is PRC and $D/K$ is totally real, there exists a $K$-place $\varphi\colon D\to K\cup\{\infty\}$, unramified in $DM$. Extend $\varphi$ to an $M$-place of $DM=FM$ and let $\varphi$ be its restriction to $F$. By Proposition~\ref{field-crossing(a)}(b) and since $F\cap D = E$,\  $\varphi^*\colon\Gal(K)\to\Gal(F/E)$ is an epimorphism. Thus, by Lemma~\ref{lem; ring cover1} with respect to $\varphi, [F:E]=[\overline{F}:\overline{E}]$.
\end{proof}

\begin{cor}
Every $p$-Hilbertian PRC field $K$ with $\Gal(K)$ pro-$p$
is strongly $p$-Hilbertian.
\end{cor}

\begin{proof}
If $K$ admits no orderings, then $K$ is PAC and 
the claim follows by Theorem~\ref{PAC p Hilbertianity}. 
Otherwise, $2 | \# \Gal(K)$, hence $p=2$.
Thus, if $E$ is the fixed field in $K(t)_s$ of some $p$-Sylow subgroup of $\Gal(K(t))$, then $[E:K(t)]$ is odd. By \cite[(1.26)]{Prestel}, the extension $E/K(t)$ is totally real.
In particular, $E/K$ is totally real.
Thus, the claim follows by Theorem~\ref{PRC p Hilbertianity}.
\end{proof}




\begin{thebibliography}{10}
\bibitem[Ax]{Ax}
  J. Ax, 
  \emph{The elementary theory of finite fields}, Annals of Mathematics \textbf{88} (1968), 239--271.
\bibitem[BSP]{BSP}
L. Bary-Soroker and E. Paran,
\emph{Fully Hilbertian fields}, Israel Journal of Mathematics \textbf{194} (2013), 507-538 (2013).


\bibitem[DH]{DH}
P. D\`ebes and D. Haran,
\emph{Almost hilbertian fields}, Acta Arithmetica \textbf{88} (1999), 269--287.


\bibitem[ELW]{ELW}
  R. Elman, T. Y. Lam and A. R. Wadsworth,
  \emph{Orderings under Field Extensions}, Journal f\"ur die reine und angewandte Mathematik \textbf{306} (1979), 7--27.

\bibitem[FLP]{FLP}
 A. Fehm, F. Legrand and E. Paran,
 \emph{Embedding problems for automorphism groups of field extensions}, Bulletin of the LMS \textbf{51} (2019), 732--744.

\bibitem[FSS]{FSS}
  B. Fein, D. Saltman and M. Schacher,
   \emph{Brauer-Hilbertian fields}, Transactions of the AMS \textbf{334} (1992), 915--928.
   
\bibitem[F]{F}
  S. Fried, 
  \emph{$\mathcal{E}$-Hilbertianity and quasi-formations}, Ph.D. Thesis, Tel Aviv University, 2017.

\bibitem[FH1]{FH1}
  S. Fried and D. Haran,
  \emph{Quasi-formations}, Israel Journal of Mathematics \textbf{229} (2019), 193--217.
\bibitem[FH2]{FH2}
  S. Fried and D. Haran,
  \emph{$\Theta$-Hilbertianity}, Journal of Algebra \textbf{555} (2020), 36--51.
\bibitem[FHV]{FHV}
  M. D. Fried, D. Haran and H. V\"olklein,
  \emph{Real Hilbertianity and the field of totally real numbers}, Contemporary Mathematics \textbf{174} (1994), 1--34.

\bibitem[FJ]{FA}
  M. D. Fried and M. Jarden,
  \emph{Field Arithmetic}, Ergebnisse der Mathematik (3) \textbf{11}, 3rd edition, Springer, 2008.
  
 \bibitem[FJ1]{FJSigma}
  M. D. Fried and M. Jarden,
  \emph{On $\Sigma$-Hilbertian fields}, 
  Pacific Journal of Mathematics \textbf{185} (1998), 307--313.

\bibitem[FV]{FV}
  M. D. Fried and H. V\"olklein,
  \emph{The embedding problem over a Hilbertian PAC-field}, Annals of Mathematics \textbf{135} (1992), 469--481.

\bibitem[GPR]{GPR}
  B. W. Green, F. Pop and P. Roquette,
  \emph{On Rumely's local global principle}, Jahresbericht DMV \textbf{97} (1995), 43--74.

\bibitem[Hil]{Hilbert}
  D. Hilbert,
  \emph{Ueber die Irreducibilit\"at ganzer rationaler Functionen mit ganzzahligen Coefficienten}, Journal f\"ur die reine und angewandte Mathematik \textbf{110} (1892), 104--129.

\bibitem[HJ]{HJ-prc}
  D. Haran and M. Jarden,
  \emph{The absolute Galois group of a pseudo real closed field}, Annali della Scuola Normale Superiore - Pisa, Serie IV, \textbf{12} (1985), 449--489.
\bibitem[Jar1]{Jarden-e-fold}
 M. Jarden,
  \emph{The elementary theory of large $e$-fold ordered fields}, Acta mathematica, \textbf{149} (1982), 239--260.
\bibitem[Jar2]{p-Hilbert}
  M. Jarden,
  \emph{$p$-Hilbertianity}, unpublished manuscript, (2004).
\bibitem[Jar3]{Patching}
  M. Jarden,
  \emph{Algebraic patching}, Springer, 2011.


\bibitem[Pr]{Prestel}
  A. Prestel,
  \emph{Lectures on formally real fields}, IMPA Publications: Lecture Notes \textbf{1093} (1984), Springer Verlag.
\bibitem[RZ]{RZ}
  L. Ribes and P. Zalesskii,
  \emph{Profinite groups}, Ergebnisse der Mathematik III \textbf{40}, 2nd edition, Springer, Berlin, 2010.
\end{thebibliography}
\end{document}